\documentclass[10pt]{amsart}
\usepackage{amssymb}
\usepackage{esint}

\usepackage{color}
\def\R{\mathbb{R}}

\def\p{\phi}
\def\vp{\varphi}
\def\eto{\eta^\circ}
\def\po{\phi^\circ}
\def\pso{\psi^\circ}
\def\Div{\textup{div}}
\def\dist{\textup{dist}}

\def\e{\varepsilon}

\def\wtos{\stackrel{*}{\rightharpoonup}}
\newcommand{\Id}{\textit{Id}}
\renewcommand{\H}{\mathcal{H}}
\newcommand{\beq}{\begin{equation}}
\newcommand{\eeq}{\end{equation}}
\newcommand{\pa}{\partial}
\newcommand{\noop}[1]{} 
\theoremstyle{plain}
\newtheorem{theorem}{Theorem}[section]
\newtheorem{proposition}[theorem]{Proposition}

\newtheorem{lemma}[theorem]{Lemma}
\newtheorem{definition}[theorem]{Definition}
\theoremstyle{remark}
\newtheorem{remark}[theorem]{Remark}

\newcommand{\dd}{d}
\newcommand{\N}{\mathbb N}
\newcommand{\newatop}{\genfrac{}{}{0pt}{1}}

\numberwithin{equation}{section}

\title[Crystalline limits of anisotropic curvature flows]{Generalized crystalline evolutions as limits of flows with smooth anisotropies}
\author[A. Chambolle \and M. Morini \and M. Novaga\and M. Ponsiglione]{Antonin Chambolle \and Massimiliano Morini \and Matteo Novaga\and Marcello Ponsiglione}

\begin{document}

\begin{abstract} 
We prove existence and
uniqueness of weak solutions to anisotropic and crystalline mean curvature flows, obtained as limit of the viscosity solutions to flows with smooth anisotropies.


\end{abstract}

\maketitle

\tableofcontents

\bibliographystyle{plain}

\section{Introduction}
In this note we deal with anisotropic, and possibly  crystalline, mean curvature flows, that is,  flows of sets $t\mapsto E(t)$ governed by the law 
\begin{equation}\label{oee} 
V(x,t) = -\psi(\nu^{E(t)}) (\kappa^{E(t)}_{\p}(x) + g(x,t)),
\end{equation}
where  $V(x,t)$ stands for the outer normal velocity of the boundary $\pa E(t)$ at $x$, $\p$ is a given norm on $\R^N$ representing the {\it surface tension}, $\kappa^{E(t)}_{\p}$ is the {\em anisotropic mean curvature} of $\pa E(t)$ associated with the anisotropy $\p$,  $\psi$ is a norm  evaluated at the outer unit normal $\nu^{E(t)}$ to $\pa E(t)$, and $g$ is a forcing term. The factor $\psi$ plays the role of a {\em mobility}. 

We refer to~\cite{CMNP17} for the motivations to study this flow, which originate in problems from phase transitions and materials science (see for instance \cite{Taylor78, Gurtin93} and references therein). Its  mathematical well-posedness is established in the smooth setting, that is when  {$\p$, $\psi$, $g$} and the initial set are sufficiently smooth and $\p$ satisfies suitable ellipticity conditions.  However, it is also well-known that in dimensions $N\geq 3$  singularities may form in finite time even in the smooth case and for regular initial sets.  When this occurs the strong formulation of  \eqref{oee}  ceases to be meaningful and thus needs to be replaced by  weaker notions of global-in-time solution.

Among the different weak approaches that have been proposed in the literature for the classical mean curvature flow (and for several other ``regular'' flows) here we recall   the so-called {\em level set formulation}~\cite{OS,EvansSpruckI,EvansSpruckII, CGG, GigaBook} and the {\em flat flow formulation} proposed by Almgren, Taylor and Wang  \cite{ATW} and based on the {\em minimizing movements} variational scheme (referred to as the ATW scheme).

However, when the anisotropy $\p$ in \eqref{oee} is non-differentiable or crystalline, the lack of smoothness of the involved differential operators makes it much harder to  pursue   the aforementioned approaches. In fact, in the crystalline case the problem of finding a suitable weak formulation of \eqref{oee}  in dimension $N\geq 3$ leading to a unique global-in-time solution for general initial sets has remained open until the very recent works \cite{CMP4, GigaPozar,  CMNP17, GigaPozar17}. 

We refer also to \cite{GiGuMa98, CaCha, BelCaChaNo} for previous  results holding for  special classes of initial data, and  to \cite{GigaGigaPozar}   for  a well-posedness result dealing with a very specific anisotropy. The two-dimensional case is somehwat easier and has been essentially settled in \cite{GigaGiga01} (when $g$ is   constant) by developing a crystalline version of the viscosity approach for the level-set equation, see also ~\cite{Taylor78,AlmTay95, AngGu89, GigaGiga98, GiGu96} for relevant former work. We also mention the recent papers
\cite{ChaNov-Crystal15, MNP}, where short time existence and uniqueness
of strong solutions for initial ``regular'' sets
(in a suitable sense) is  shown.
\smallskip

Let us now briefly describe the most recent progress on the problem. 
In \cite{CMP4},  the first  global-in-time existence and uniqueness result for the level set flow associated to \eqref{oee},  valid in all dimensions, for arbitrary (possibly unbounded) initial sets, and for general (including crystalline) anisotropies $\p$ {was} established,  but under the particular choice $\psi=\p$ (and $g=0$). The main contribution of that work is the observation that
 the variant of the ATW scheme proposed in \cite{Chambolle, CaCha}) converges to solutions that  satisfy a  new stronger {\em distributional}  formulation of the problem in terms of distance functions. Such a formulation is only reminiscent of, but not quite the same as, the distance formulation studied in \cite{Soner93} (see also \cite{BaSoSou,AmbrosioSoner,CaCha, AmbrosioDancer}), and because of its distributional character it   enables the use of parabolic  PDE's arguments in order to establish a  comparison result yielding uniqueness. 

In \cite{CMNP17}, we first observe that the methods of \cite{CMP4} can be pushed to treat  bounded spatially Lispchitz continuous forcing terms $g$ and more general mobilities $\psi$, which  are  ``regular'' with respect to the anisotropy $\p$. More precisley,  a norm $\psi$ is said to be $\p$-regular if the associated $\psi$-Wulff shape $W^\psi$  satisfies a uniform inner $\p$-Wulff shape condition at all points of its boundary. Such a condition implies that the $\p$-curvature $k_\p$ of $\pa W^\psi$ is bounded above and it enables us to show that a distributional formulation in the spirit of \cite{CMP4} still holds true. Next, owing to the simple   observation  that the  $\p$-regular mobilities are dense, we succeed  in extending the notion of solution to general mobilities by an approximation procedure. More precisely,  
by establishing delicate stability estimates on the ATW scheme, we show that if $\psi$ is any norm and $\psi_n\to \psi$, with $\psi_n$ a $\p$-regular mobility for every $n$, then the corresponding distributional level set solutions $u^{\psi_n}$, with the given initial datum $u^0$,  admit a unique limit $u^\psi$ (independent of the choice of the approximating ${\psi_n}$), which we may therefore regard as the unique solution to the level set flow with mobility $\psi$ and initial datum $u^0$. 
As a byproduct of this analysis, we also settle the problem of the uniqueness (up to fattening) of flat flows for general mobilities. Once again, our results hold   in all dimensions, for arbitrary (possibly unbounded) initial sets and  general, possibly  crystalline anisotropies $\p$. 

By completely different methods, in \cite{GigaPozar} and  more recently in \cite{GigaPozar17}, the authors succeed in extending the viscosity approach of \cite{GigaGiga01} to the case $N=3$ and to the general case $N\geq 3$, respectively. In fact, as in \cite{GigaGiga01} they  are able to deal with very general equations of the form 
 $$
 V=f(\nu,-\kappa^E_\p)\,,
 $$
  with $f$ continuous and non-decreasing with respect to the second variable,
but without spatial dependence, 
establishing existence and uniqueness for the corresponding level set formulation. 
However,  their method, as far as we know,  works only for 
purely crystalline anisotropies $\p$, bounded initial sets, and constant forcing terms.
\smallskip

Here we propose a variant  of the approach of \cite{CMNP17}, by deriving
  existence,
uniqueness and some properties of anisotropic and crystalline flows
directly from the corresponding properties of smooth (i.e., with smooth anisotropies) flows,
appropriately defined as viscosity solutions of a geometric
PDE.  
This leads to a more direct and easier proof of the well-posedness of \eqref{oee} for general mobilities and anisotropies, relying on purely viscosity methods, which however does not 
provide any information about the uniqueness of flat flows.

Let us describe the new approach in more detail.  The starting point is the observation that when the anisotropy is smooth, the distributional formulation of \cite{CMP4, CMNP17} is equivalent to the classical viscosity formulation, see Section~\ref{sec:vs}. Next, in Section~\ref{subsec:levelset} we show that if $\phi_n\to  \p$, with $\phi_n$ smooth, and if $\psi_n\to \psi$, with $\psi_n$  $\phi_n$-regular ``uniformly'' with respect to $n$ (see the statement of Theorem~\ref{th:exist} below for the precise meaning), then
the corresponding viscosity (and thus distributional) level set solutions $u_n$ converge locally uniformly to the unique distributional level set flow with anisotropy $\p$ and  ($\p$-regular) mobility $\psi$. This leads to a new proof of the existence of distributional level set solutions for $\p$-regular mobilities, without using the ATW scheme as in \cite{CMNP17}.

In Sections~\ref{sec:forcing}~and~\ref{sec:mobilities} we establish the crucial stability estimates of the flow with respect to changing $\p$-regular mobilities. This is achieved once again by exploiting the viscosity formulation in order to prove first the estimates in the case of smooth anisotropies and  to conclude by approximation.

Finally, in Section~\ref{sec:ex} we prove the main existence and uniqueness result for the level set formulation of \eqref{oee}, in the case of general anisotropies and mobilities. 
In this last step we proceed essentially as in \cite{CMNP17}: we approximate any mobility $\psi$ by a sequence $\p$-regular mobilities $\psi_n$ and  show, by means of the stability estimates of the previous sections, that the corresponding solutions admit a unique limit. 

\section*{Acknowledgements}
A. Chambolle was partially supported by the ANR projects ANR-12-BS01-0014-01 ``GEOMETRYA" and ANR-12-BS01-0008-01 ``HJnet". M. Novaga was partially supported by 
by the GNAMPA and by the University of Pisa via grant PRA 2017
``Problemi di ottimizzazione e di evoluzione in ambito variazionale".
M. Morini and M. Ponsiglione were partially supported by the GNAMPA 2016 grant ``Variational methods for nonlocal geometric flows".

\section{Distributional mean curvature flows}

Given a norm $\eta$ on $\R^N$
(a convex, even, one-homogeneous real-valued function with $\eta(\nu)>0$
if $\nu\neq 0$), we define a polar norm $\eto$ by
$\eto(\xi):=\sup_{\p(\nu)\le 1}\nu\cdot\xi$ 
and an associated anisotropic perimeter $P_\eta$ as
\[
P_\eta(E):=\sup\biggl\{\int_E\Div \zeta\, dx: \zeta\in C^1_c(\R^N; \R^N),\, \eto(\zeta)\leq 1\biggr\}\,.
\]
As is well known, $(\eto)^\circ=\eta$ so that when the set $E$ is smooth enough
one has
\[
P_\eta(E) = \int_{\partial E}\eta(\nu^E)d\H^{N-1}\,,
\]
which is the perimeter of $E$ weighted by the surface tension $\eta(\nu)$.

We will make repeated use of the following identities 
\beq\label{subp}
\partial\eta(\nu) =\{\xi\,:\,\eto(\xi)\le1\textup{ and }\xi\cdot \nu \ge \eta(\nu)\}= \{\xi\,:\,\eto(\xi)=1\textup{ and }\xi\cdot \nu=\eta(\nu)\},
\eeq
(and the symmetric statement for $\eto$) for $\nu\neq 0$.
Moreover,  $\partial\eta(0)=\{\xi\,:\,\eto(\xi)\le 1\}$ while $\partial\eto(0)=\{\xi\,:\,\eta(\xi)\le 1\}$.
For $R>0$ we denote
$$
W^{\eta}(x,R):=\{y\,:\,\eto(y-x)\le R\}\,.
$$ 
Such a set is called the {\em Wulff shape} (of radius $R$ and center $x$) associated with the norm $\eta$ and represents
the unique (up to translations) solution of the anisotropic isoperimetric problem 
$$
\min\left\{P_\eta(E)\,:\,  |E|=|W^{\eta}(0,R)|\right\},
$$
see for instance~\cite{FonsecaMuller}.


We denote by $\dist^\eta(\cdot, E)$ the  distance from $E$ induced by the norm $\eta$, that is, for any $x\in \R^N$
\begin{equation}\label{polardist}
\dist^\eta(x, E):=\inf_{y\in E}\eta(x-y)
\end{equation}
if $E\neq\emptyset$ and  $\dist^\eta(x,\emptyset):= + \infty$.  
Moreover,  we denote by $\dd^\eta_E$ the signed distance from $E$ induced by $\eta$, i.e., 
\[
\dd^\eta_E(x):= \dist^\eta(x,E) - \dist^\eta(x,E^c)\,.
\]
so that $\dist^\eta(x,E)=\dd^\eta_E(x)^+$ and $\dist^\eta(x,E^c)=\dd^\eta_E(x)^-$, where we adopted  the  standard notation  $t^+:=t\lor  0$ and $t^-:=(-t)^+$).
Note that by \eqref{subp} we have $\eta(\nabla d^{\eto}_{E})=\eto(\nabla d^{\eta}_{E})=1$ a.e.~in $\R^N\setminus \partial E$.

Finally we recall that a sequence of closed sets $E_n$ in $\R^m$ converges to a closed set $E$ in the {\em Kuratowski sense}
if the following conditions are satisfied
\begin{itemize}
\item[(i)] if $x_n\in E_n$, any limit point of $\{x_n\}$ belongs to $E$;
\item[(ii)] any $x\in E$ is the limit of a sequence $\{x_n\}$, with $x_n\in E_n$.
\end{itemize}
 and we write
$$
E_n\stackrel{\mathcal K}{\longrightarrow} E\,.
$$
Since $E_n\stackrel{\mathcal K}{\longrightarrow} E$ if and only if
(for any norm $\eta$) $\dist^{\eta}(\cdot, E_n)\to \dist^{\eta}(\cdot, E)$ locally uniformly in $\R^m$, by Ascoli-Arzel\`a's Theorem any sequence of closed sets admits a  converging subsequence  in the Kuratowski sense.

\subsection{The weak formulation  of the crystalline flow}\label{stass}

In this section we recall  the weak   formulation of the crystalline mean curvature flow introduced in \cite{CMP4, CMNP17}.


In what follows, we will consider  forcing terms $g:\R^N\times [0, +\infty)\to \R$ satisfying the following two hypotheses:
\begin{itemize}
\item[H1)] for every $T>0$, $g\in L^{\infty}(\R^N\times (0, +\infty))$;
\item[H2)] there exists $L>0$ such that $g(\cdot, t)$ is L-Lipschitz continuous with respect to the metric $\pso$ for a.e. $t>0$. Here $\psi$ is the norm representing the mobility in \eqref{oee}.
\end{itemize}
\begin{remark}
Assumption H1) can be in fact weakened and replaced by 
\begin{itemize}
\item[H1)'] for every $T>0$, $g\in L^{\infty}(\R^N\times (0, T))$.
\end{itemize}
 Indeed under the weaker assumption H1)', all the arguments and the estimates presented throughout the paper continue to work in any  time interval $(0,T)$, with some of the constants involved possibly depending on $T$. {In the 
same way, if one restricts our study to the evolution of sets with
compact boundary, then one could assume that $g$ is only locally
bounded in space.}
We assume H1) instead of H1)' only to simplify the presentation. 
\end{remark}

Let   $\p$, $\psi$ be two (possibly crystalline) norms representing the anisotropy and the mobility in \eqref{oee}, respectively. We recall the following distributional formulation of \eqref{oee}.
\begin{definition}[See \cite{CMNP17}]\label{Defsol}
Let $E^0\subset\R^N$ be 
a closed set. 
Let $E$ be a closed set in $\R^N\times [0,+\infty)$ and
for each $t\geq 0$ denote $E(t):=\{x\in \R^N\,:\, (x,t)\in E\}$. We
say that $E$ is a {\em superflow} of \eqref{oee} with
initial datum $E^0$ if
\begin{itemize}
\item[(a)] {\sc Initial Condition:} $E(0)\subseteq {E}^0$;
\item[(b)] {\sc Left Continuity:}
$E(s) \stackrel{\mathcal K}{\longrightarrow} E(t)$ as $s\nearrow t$ for all $t>0$;
\item[(c)] {\sc Extinction Time:} 
If for  $t\ge 0$, ${E}(t)=\emptyset$, then $E(s)=\emptyset$ for all $s > t$;
\item[(d)] {\sc Differential Inequality:} 
Set $T^*:=\inf\{t>0\,:\, E(s)=\emptyset \text{ for $s\geq t$}\}$, 
and 
$$
d(x,t):=\dist^{\pso}(x, E(t)) \qquad \text{ for all } (x,t)\in \R^N\times (0,T^*)\setminus E.
$$
Then there exists $M>0$ such that the inequality
\begin{equation}\label{eq:supersol}
 \partial_t d \ge \Div z+g- Md
\end{equation}
holds in the distributional sense in $\R^N\times (0,T^*)\setminus E$
for a suitable $z\in L^\infty(\R^N\times (0,T^*))$ such that
$z\in \partial\p(\nabla d)$~a.e., $\Div z$ is a Radon measure in $\R^N\times (0,T^*)\setminus E$, and 
 $(\Div z)^+\in L^\infty(\{(x,t)\in\R^N\times (0,T^*):\, d(x,t)\geq\delta\})$ for every $\delta\in (0,1)$.
\end{itemize}

We say that $A$, open set in $\R^N\times [0,+\infty)$, is
a subflow of \eqref{oee} with initial datum $E^0$ if $A^c$ is a superflow of \eqref{oee} with $g$ replaced by $-g$ and with initial datum $(\mathring{E}^0)^c$.

Finally, we say that $E$, closed set in $\R^N\times [0,+\infty)$,  is a solution of \eqref{oee} with initial datum $E^0$ if it is a superflow and if $\mathring{E}$ is a subflow, both with initial datum $E^0$.
\end{definition}

It is shown in~\cite{CMNP17} (also~\cite{CMP4} for a simpler equation),
using quite standard parabolic comparison arguments,
that such evolutions satisfy a comparison principle:
\begin{theorem}[{\cite[Thm~2.7]{CMNP17}}]\label{thm:uniq}
Let $E$ be a superflow  with initial datum $E^0$ and
$F$ be a subflow with initial datum $F^0$ in the sense of Definition~\ref{Defsol}. Assume that
$\dist^{\pso}(E^0,{F^0}^c)=:\Delta>0$. 
Then, 
$$\dist^{\pso}(E(t),F^c(t))\ge \Delta \mathrm{e}^{-Mt} \qquad \text{ for all } t\ge 0,$$
where $M>0$ is as in 
\eqref{eq:supersol} for both $E$ and $F$.
\end{theorem}

We now recall the corresponding notion of sub- and supersolution to the level set flow associated with \eqref{oee}, see again \cite{CMNP17}.

 \begin{definition}[Level set subsolutions and supersolutions]\label{deflevelset1}
 Let  $u^0$ be a  uniformly continuous function on $\R^N$. We will say that a  lower semicontinuous   function $u:\R^N\times [0, +\infty)\to \R$ is a {\em level set supersolution} corresponding to \eqref{oee}, with initial datum $u^0$,  if $u(\cdot, 0)\geq u^0$ and if for a.e. $\lambda \in \R$ the closed sublevel set $\{(x,t)\,:\, u(x,t)\leq \lambda\}$  is a superflow {of}  \eqref{oee} in the sense of Definition~\ref{Defsol}, with initial datum $\{u_0\leq \lambda\}$.
 
 We will say that an  upper-semicontinuous   function $u:\R^N\times [0, +\infty)\to \R$ is a {\em level set subsolution} corresponding to \eqref{oee}, with initial datum $u^0$, if~$-u$ is a superlevel set flow  in the previous sense, with initial datum $-u_0$ and with $g$ replaced by~$-g$.
 
 Finally, we will  say that a continuous   function $u:\R^N\times [0, +\infty)\to \R$ is a {\em  solution} to the level set flow corresponding to \eqref{oee} if it is both a level set subsolution and  supersolution. 
 \end{definition}

As shown in \cite{CMNP17}, Theorem~\ref{thm:uniq} easily yields that almost all closed sublevels of a solution of the level set flows are solutions of \eqref{oee} in the sense of Definition \eqref{Defsol}. Moreover,  the following comparison principle  between level set subsolutions and supersolutions holds true.

\begin{theorem}[{\cite[Thm~2.8]{CMNP17}}]\label{th:lscomp}
Let $u^0$, $v^0$ be uniformly continuous functions on $\R^N$  and let $u$, $v$ be respectively a level set subsolution with initial datum $u^0$ and a level set supersolution with initial datum $v^0$, in the sense of Definition~\ref{deflevelset1}. If $u^0\leq v^0$, then $u\leq v$.
\end{theorem}

For smooth anisotropies, solutions to the level set flow and (minus the characteristic function of) solutions of the geometric flow in the sence of Definition \ref{Defsol} are in fact viscosity solutions of the (degenerate) parabolic equation \eqref{eq:viscoflow} below. This classical fact will be shown and exploited  to some extent to non smooth anisotropies in the next two sections.

\subsection{Viscosity solutions}\label{sec:vs}
We show here that in the smooth cases, the notion of solution
in Definition~\ref{Defsol}
coincides with the definition of standard viscosity solutions for
geometric motions, as for instance
in~\cite{BarlSouga98}. This property will be  helpful to establish estimates using standard approaches
for viscosity solutions.

\begin{lemma}\label{lem:visco}
Assume that $\p,\psi,\pso\in C^2(\R^N\setminus\{0\})$, and that $g$ is continuous.
Let $E$ be a superflow in the sense of Definition~\ref{Defsol}.
Then, $-\chi_E$ is a viscosity supersolution of
\begin{equation}\label{eq:viscoflow}
u_t = \psi(\nabla u)\big( \Div\nabla \p(\nabla u)+g\big).
\end{equation}
in $\R^N\times (0,T^*]$ where $T^*$ is the possible extinction time of $E$.

Conversely, a viscosity supersolution
$-\chi_{E(t)}$ of~\eqref{eq:viscoflow}
defines a superflow in the sense of Definition~\ref{Defsol}.
\end{lemma}
\begin{proof} A similar statement (in a simpler context) is proved
in~\cite[Appendix]{CMP4}, while it is proved in~\cite{CMNP17}
that a superflow defines a viscosity supersolution. We therefore
here focus on the converse: 
Given an evolving set $E(t)$ such
that $-\chi_E$ is a viscosity supersolution of~\eqref{eq:viscoflow},
we show that $E(t)$ is a superflow in the sense of
Definition~\ref{Defsol} (with  the constant $M$ in~\eqref{eq:supersol}
equal to the Lipschitz constant $L$ of $g(\cdot,t)$ appearing in the assumption H2)).

\noindent\textit{Step 1: Left Continuity and Extinction Time.}
Let $T^*\in [0,+\infty]$ be the (first) extinction time of $E$, and assume without loss of generality $T^*>0$.
Let $d(x,t):=\dist^{\pso}(x,E(t))$. We fix $\delta>0$ and
we set  $A=(\R^N\times [0,T^*))\setminus E$ and $A^\delta=A\cap\{d>\delta\}$.
Let $(x,t)$ with $d(x,t)=R>0$. Then $W^\psi(x,R-\e)\cap E(t)=\emptyset$
 for any $\e>0$ (small). There exists a constant $C$ (depending on $\phi,\psi$)
 such that, letting $W(s)=\R^N\setminus W^\psi(x,R-\e-(C/R+\|g\|_\infty)s)$,
 $-\chi_{W(s)}$ is a viscosity subsolution of~\eqref{eq:viscoflowdist},
 for $s\le R^2/(2(C+R\|g\|_\infty))$ and $\e\le R/4$.
 By standard comparison results~\cite{BaSoSou}, it follows that
 $E(t+s)\subset W(s)$ for such times $s$, so that
 $d(x,t+s)\ge R-\e-(C/R+\|g\|_\infty)s$. Hence, letting $\e\to 0$,
 we find that
\begin{equation}\label{sc1}
d(x,t+s)\ge d(x,t)-(C/\delta+\|g\|_\infty)s \qquad \text{ if } (x,t), \,(x,t+s)\in A^\delta.
\end{equation}
 In particular, it follows that $\partial_t d$
 is bounded from below in such sets and hence is a measure.
By \eqref{sc1} and the fact that $E$ is closed  we deduce that the left continuity (b) of Definition~\ref{Defsol}
 holds for $E(t)$.
 Moreover,  the same argument shows that if $t>T^*$ then $d(x,t)=+\infty$,
 showing also point~(c).

\noindent \textit{Step 2: The distance function is a viscosity supersolution.}
We now
show that the function $d(x,t)$ is
a viscosity supersolution of
\begin{equation}\label{eq:viscoflowdist}
u_t = \psi(\nabla u)\big( D^2\p(\nabla u):D^2u +g - L u\big).
\end{equation}
In fact, this is essentially classical~\cite{Soner93}, however
the proof in this reference needs to be adapted to deal with the
forcing term. An elementary proof is as follows: let $\eta$ be
a smooth test function and assume
$(\bar x,\bar t)$ is a contact point, where $\eta(\bar x,\bar t)=d(\bar x,
\bar t)$ and $\eta\le d$. If the common value of $\eta,d$ at $(\bar x,\bar t)$
is zero then it is also a contact point of $1-\chi_E$ and $\eta$,
so that
\begin{equation}\label{eq:superviscodist}
\partial_t\eta(\bar x,\bar t)\ge
\psi(\nabla \eta(\bar x,\bar t))\big(
D^2\p(\nabla\eta(\bar x,\bar t)):D^2\eta(\bar x,\bar t)+g(\bar x,\bar t)-
L\eta(\bar x,\bar t)\big)
\end{equation}
obviously holds, by definition (recalling \eqref{eq:viscoflow} and that $\eta(\bar x,\bar t)=0$). Hence we consider the case 
where $R=d(\bar x,\bar t)>0$. Let $\bar y\in\partial E(\bar t)$
such that $R=\pso(\bar x-\bar y)$.
We let
\[
\eta'(y,t):= \eta(y+\bar x-\bar y,t)-R
\le d(y+\bar x-\bar y,t)-R
\le d(y,t)
\]
since $d$ is $1$-Lipschitz in the $\pso$ norm.
In particular, in a neighborhood of $(\bar y,\bar t)$,
$\eta'(y,t)\le 1-\chi_{E(t)}(y)$.
On the other hand,
$\eta'(\bar y,\bar t)=0=d(\bar y,\bar t)=1-\chi_{E(\bar t)}(\bar y)$.
Hence, by \eqref{eq:viscoflow} 
\[
\partial_t\eta(\bar x,\bar t)=\partial_t\eta'(\bar y,\bar t)\ge
\psi(\nabla \eta'(\bar y,\bar t))\big(
D^2\p(\nabla\eta'(\bar y,\bar t)):D^2\eta'(\bar y,\bar t)+g(\bar y,\bar t)
\big).
\]
Since $g(\bar y,\bar t)\ge g(\bar x,\bar t)-L\eta(\bar x,\bar t)$,
\eqref{eq:superviscodist} follows.

\noindent \textit{Step 3: Differential inequality.}
A  classical remark is that $d^2$, as an infimum of
the uniformly semiconcave functions $\pso(\cdot -y)^2$, $y\in E(t)$,
is semiconcave, hence in $A^\delta$ one has $D^2d\le C/\delta I$ in the sense of measures
for some constant $C$ depending only on $\pso$.
In particular,
$\Div\nabla\p(\nabla d)=D^2\p(\nabla d):D^2d \le C/\delta$
in $A^\delta$ in the sense of measures.

We  proceed as in~\cite{CMP4}: for $n\ge 1$, let $d_n(x,t):=
\min_{s} (d(x,t-s)+n s^2)$ which is semiconcave and converges
to $d$ as $n\to\infty$. Moreover, one can easily check that
$d_n(\cdot,t)\to d(\cdot,t)$ locally uniformly if $t$ is a 
continuity point of $d$. 
Let $B\subset A^\delta$ be an open  ball,
(where in particular $d$ is bounded from below
by $\delta$ and from above), and observe that $d_n$ is still a supersolution
of~\eqref{eq:viscoflowdist}, provided $g(x,t)$ is replaced
with $g(x,t)-\omega_n)$ for some $\omega_n\to 0$ as $n\to +\infty$. 
Since $d_n$, which is semi-concave,
has a second-order jet almost everywhere in $B$,
equation~\eqref{eq:viscoflowdist} holds for $d_n$ almost everywhere in $B$.
Reasoning as in~\cite[Appendix]{CMP4}, we deduce that
\begin{equation}\label{eq:approxv}
\partial_t d_n \ge \psi(\nabla d_n)\big(
\Div z_n + g - \omega(C/n) - L d_n\big).
\end{equation}
in the distributional sense (or as measures) in $B$, where
$z_n:=\nabla\p(\nabla d_n)$. It remains to send $n\to\infty$:
clearly, $\partial_t d_n\to\partial_t d$ in the distributional
sense. Consider $(x,t)$ a point where $\nabla d(x,t)$,
and $\nabla d_n(x,t)$ exist for all $n$.
First, if $d(x,t-s)+ns^2$ attains the minimum in $s_n$,
one has for any $p\in \partial^+d(x,t-s_n)$ (the spatial
supergradient of the semiconcave function $d(\cdot,t-s_n)$) that
\begin{multline*}
d_n(x+h,t)\le d(x+h,t-s_n)+n s_n^2 
\\
\le d(x,t-s_n)+p\cdot h + \frac{C}{\delta}|h|^2 +n s_n^2 = d_n(x,t)+ p\cdot h+\frac{C}{\delta}|h|^2
\end{multline*}
showing that $p\in\partial^+d_n(x,t)=\{\nabla d_n(x,t)\}$.
We deduce that $d(\cdot,t-s_n)$ is differentiable at $x$,
with gradient $\nabla d_n(x,t)$, and in particular that
$\psi(\nabla d_n(x,t))=1$.

Assume now that in addition $d$ is continuous at $t$.
Then $d_n(\cdot,t)\to d(\cdot,t)$
uniformly in $B\cap (\R^N\times\{t\})$, and using the (uniform) semiconcavity of these
functions one also deduces that $\nabla d_n(x,t)\to\nabla d(x,t)$
a.e.: hence, $z_n(x,t)=\nabla\p(\nabla d_n(x,t))$
converges to $z(x,t)=\nabla\p(\nabla d(x,t))$ almost everywhere.
Hence we may send $n$ to $\infty$ in~\eqref{eq:approxv} to find
that 
\[
\partial_t d \ge \Div z + g - L d
\]
in the distributional sense in $B$, with $z=\nabla\p(\nabla d)$ a.e.

This shows the Lemma.
\end{proof}

\subsection{The level set formulation}\label{subsec:levelset}

Let $u^0\in BUC(\R^N)$ (a Bounded, Uniformly Continuous function).
Then, it is well known~\cite{CGG} that if $\p\in C^2(\R^N\setminus\{0\})$,
$\psi,g$ are continuous,
there exists a unique viscosity solution $u$ of~\eqref{eq:viscoflow} with initial datum $u^0$.
Moreover, for all $\lambda\in\R$,
$-\chi_{\{u<\lambda\}}$ is a viscosity supersolution
and $-\chi_{\{u\le\lambda\}}$ a viscosity subsolution of the same equation.
If in addition $\psi,\pso\in C^2(\R^N\setminus\{0\})$,
it follows from Lemma~\ref{lem:visco}  
that 
$E_\lambda(t):=\{u(\cdot,t)\le\lambda\}$
is a superflow in the sense of Definition~\ref{Defsol},
while $A_\lambda(t):=\{u(\cdot,t)<\lambda\}$ is a 
subflow\footnote{In case of ``fattening'', also $\overline{\{u<\lambda\}}$
is a superflow, and the interior of $\{u\le\lambda\}$ a subflow.}.

In what follows we will say that a given norm $\eta$ is {\em smooth and elliptic} if both $\eta$ and $\eta^{\circ}$ belong to $C^2(\R^N\setminus\{0\})$.

We now consider sequences $\p_n,\psi_n$   of smooth and elliptic
anisotropies/mobilities converging
to $\p,\psi$. We also consider $g_n(x,t)$ a smooth forcing term,
 which converges
to $g(x,t)$ weakly-$*$ in $L^\infty(\R^N\times [0,+\infty))$. We assume also that $g_n$ is uniformly spatially Lipschitz continuous and we denote by  $L$, $M$ the (uniform) Lispchitz constants of $g_n$ with respect to $\pso_n$ and $\po_n$, respectively.
Given $u_n$ the corresponding unique viscosity solution of~\eqref{eq:viscoflow}
(with $\psi_n,\p_n,g_n$ instead of $\psi,\p,g$) with initial datum $u^0$, we want to study
the possible limits of $u_n$.
If the limiting anisotropies and forcing term are still smooth enough,
it is well known that the limiting $u$ is the unique viscosity solution of the
corresponding limit problem. If not, we will show that the limit
is still unique. 
We recall (see \cite{CMNP17}) the following 
\begin{definition}\label{def:phiregular}
We will say that a norm $\psi$ is $\p$-regular if  the associated Wulff shape $W^\psi(0,1)$ satisfies a uniform interior $\p$-Wulff shape condition, that is, if  there exists $\e_0>0$ with the following property:  for every $x\in \partial W^\psi(0,1)$ there exists $y\in W^\psi(0,1)$ such that  $W^\p(y, \e_0)\subseteq W^\psi(0,1)$ and $x\in \partial W^\p(y,\e_0)$.
 \end{definition}
{Notice that it is equivalent to saying that $W^\psi(0,1)$ is the
sum of a convex set and  $W^\p(0,\e_0)$, or equivalently that
$\psi(\nu)=\psi_0(\nu)+\e_0\p(\nu)$ for some convex function $\psi_0$.}


We now show the following result.

\begin{theorem}\label{th:exist}
Let $(\psi_n)_n$, $(\p_n)_n$ and $(g_n)_n$ as above, and, in addition,  assume that the mobilities $(\psi_n)_n$ are
uniformly $\p_n$-regular, meaning that  $\e_0>0$ in the Defintion~\ref{def:phiregular} does not depend on $n$. 
Let $u_n$ be the level set solutions to \eqref{oee} in the sense of Definition~\ref{deflevelset1},  with initial datum $u^0$, anisotropy $(\psi_n)_n$, mobility $(\p_n)_n$ and forcing term $(g_n)_n$.
  Then,   $u_n$ converge locally uniformly to  the unique level set solution $u$ to \eqref{oee} in the sense of Definition~\ref{deflevelset1},  with initial datum $u^0$, anisotropy $\psi$, mobility $\p$ and forcing term $g$.
\end{theorem}

\begin{proof}
A first observation is that the functions
$u_n$ remain uniformly
continuous in space and time on $\R^N\times [0,T]$ for all $T>0$,
with a modulus depending only on the modulus of continuity $\omega$ of $u^0$
and the Lipschitz constant $M$. Indeed,  by Proposition~\ref{prop:stabsmooth} below it follows that for any $\lambda<\lambda'$
$$
\dist^{\po_n}(\{u_n(\cdot, t)\leq \lambda\},\{u_n(\cdot, t)\geq \lambda'\})\ge 
\Delta e^{-\beta Mt }\,,
$$
where $\Delta:=\omega^{-1}(\lambda'-\lambda)\geq\dist^{\po}(\{u^0\leq \lambda\},\{u^0\geq \lambda'\})>0$, and $\beta>0$ depends (for large $n$) only on $\p$ and $\psi$ (see \eqref{eq:compani0}). Therefore, $u_n(\cdot, t)$ is uniformly continuous with modulus of continuity with respect to the norm $\po_n$ given by $\omega(e^{\beta Mt }\cdot )$. 
As for the equicontinuity in time, we set $\omega_T(s):=\omega(e^{\beta MT }s)$ and we start by observing that for any $x\in \R^N$, $\e>0$, $t\in (0, T]$, and $n\in N$ we have
$$
W^{\p_n}(x, \omega_T^{-1}(\e))\subseteq\{y:u_n(y,t)>u_n(x,t)-\e\}.
$$ 
Therefore, by standard comparison results we have that $u_n(x,t')>u_n(x,t)-\e$ provided that 
$0<t'-t<\tau$, where $\tau$ is the extintion time for $W^{\p_n}(x, \omega_T^{-1}(\e))$ under the evolution \eqref{oee}. Analogously, one shows that  $u_n(x,t')<u_n(x,t)+\e$ if $0<t'-t<\tau$. Since $\tau$ is bounded away from zero by a quantity independent of $n$ 
(depending only on $\e$, $\sup_n\|g_n\|_\infty$ and, for $n$ large, on $\p$ and $\psi$), see for instance \cite[Remark 4.6]{CMNP17}. This establishes the equicontinuity in time. 
 
Hence, up to a subsequence (not relabelled), we may assume that $u_n$ converges
locally uniformly to some $u$.
In view of Theorem~\ref{th:lscomp}, it is enough to show that $u$ is  a solution in the sense of Definition~\ref{deflevelset1}, that is, that for a.e. $\lambda\in \R$ the set $E_\lambda:=\{u\le\lambda\}$  is a superflow in the sense of Definition~\ref{Defsol} and $A_\lambda:=\{u<\lambda\}$ a subflow.

We prove the assertion for $E_\lambda$. We first notice that  since $u_n\to u$ locally uniformly, the 
 Kuratowski limit superior of  the sets
$E_n:=\{u_n\le\lambda\}$ as $n\to\infty$ is contained in $E_\lambda$.

By Lemma~\ref{lem:visco}, the sets $E_n$ are superflows in the
sense of Definition~\ref{Defsol}. We consider
$d_n(x,t):=\dist^{\pso_n}(x,E_n(t))$,
$d(x,t):=\dist^{\pso}(x,E_\lambda(t))$,
the corresponding distance functions, which are finite up
to some time $T^*_n, T^*\in (0,+\infty]$ respectively, 
where $T^*$ is defined according with Definition \ref{Defsol}. Notice that $T^*$ is increasing with respect to $\lambda$, and that if $\lambda$ is a continuity point, then  we have $T^*_n\to T^*$,
as $n\to\infty$.
Reasoning as in the proof of Lemma~\ref{lem:visco}-\textit{Step 1}
we can show that $d_n(x,t+s)^2\ge d_n(x,t)^2-2Cs$ for some constant $C$
which does not depend on $n$. Indeed, $C$ is essentially the maximal
speed of the Wulff shape $W^{\psi_n}$, which is bounded by 
$\max_{\xi}\psi_n(\max_{\partial W^{\psi_n}}\kappa_{\p_n}+\|g_n\|_\infty)$. The
curvature $\kappa_{\p_n}$ of $\partial W^{\psi_n}$ is in $[0,(N-1)/\e]$, thanks
to the assumption that $\psi'_n:=\psi_n-\e\p_n$ is convex, which yields that
$W^{\psi_n}=W^{\psi'_n}+\e W^{\p_n}$.

This implies (see for instance details in
the proof of~\cite[Prop.~4.4]{CMP4})
that one can find a set 
at most countable $\mathcal{N}\subset
(0,T^*)$, such that for all $t\not\in\mathcal{N}$,
$d_n(\cdot, t)\to d(\cdot,t)$ locally uniformly.
If $B\subset\subset (\R^N\times (0,T^*))\setminus E_\lambda$,
one has $B\cap E_n=\emptyset$ for $n$ large enough and
\[
\partial_t d_n \ge \Div z_n + g_n - Ld_n
\]
in the distributional sense in $B$, thanks to~\eqref{eq:supersol} and
Lemma~\ref{lem:visco}. Here, $z_n=\nabla\p_n(\nabla d_n)$.
Notice that $z_n$ are (for $n$ large) well defined and bounded in $L^\infty(\R^N\times (0,T))$ for any $T<T^*$.
In the limit, we find that~\eqref{eq:supersol}
holds for $d$, with $z$ the weak-$*$ (local in time)
limit of $(z_n)_n$ (or rather, in fact, a subsequence).
It remains to show that $z\in\partial \p(\nabla d)$ a.e.~in $B$.
An important observation is that, using again the $\p_n$-regularity
of $\psi_n$, one can show that $\Div \nabla\p_n(\nabla d_n)\le (N-1)/(\e_0 d_n)$,
hence it is bounded in $\{d_n>\delta\}$. In particular, in the limit,
$(\Div z)^+\chi_{\{d>\delta\}}\in L^\infty(\R^N\times (0,T)^*)$.


To show $z\in\partial \p(\nabla d)$ a.e.~in $B$,
we establish that $z\cdot\nabla d\ge \p(\nabla d)$~a.e.~in $B$.
The proof here is as in~\cite{CMP4}.
There exists $\delta$ such that for all $n$ large enough, $d_n\ge \delta$
in $B$, hence $\Div z_n\le (N-1)/( \e_0 \delta)$.
Let $\eta\in C_c^\infty(B;\R_+)$, then
\[
\int_B \p(\nabla d)\eta dxdt\le\liminf_n
\int_B \p_n(\nabla d_n)\eta dxdt = \liminf_n \int_B (z_n\cdot\nabla d_n)\eta \,
dxdt.
\]
On the other hand,
\begin{equation}\label{irtp}
\int_B (z_n\cdot\nabla d_n)\eta dxdt=
\int_B (z_n\cdot\nabla d)\eta dxdt + \int_B (z_n\cdot\nabla (d_n-d))\eta \, dxdt,
\end{equation}
and $\lim_n \int_B (z_n\cdot\nabla d)\eta dxdt=\int_B (z\cdot\nabla d)\eta \,
dxdt$ since $z_n\wtos z$.

It remains to prove that the second addend in the right hand side of \eqref{irtp} tends to zero as $n\to +\infty$. Set
$$
m_n(t)=\min_{x:(x,t)\in\overline{B}}(d_n(x,t)-d(x,t)), \qquad M_n(t)=\max_{x:(x,t)\in\overline{B}}(d_n(x,t)-d(x,t)).
$$
Then for all $t\not\in\mathcal{N}$,
$M_n(t)-m_n(t)\to 0$.
One has
\begin{multline*}
\int_B (z_n\cdot\nabla (d_n-d))\eta dxdt
= \int_B(z_n\cdot\nabla (d_n-d-m_n(t)))\eta \, dxdt
\\
= -\int_B (d_n-d-m_n) \eta \, \Div z_n \,dxdt - \int_B(d_n-d-m_n)z_n\cdot\nabla\eta \, dxdt.
\end{multline*}
The last integral goes to zero as $n\to\infty$. 
Since $(d_n-d-m_n(t))\eta \ge 0$ we have
\[
-\int_B (d_n-d-m_n) \eta \, \Div z_n \, dxdt
\ge -\frac{N-1}{\e_0 \delta}
\int_B (d_n-d-m_n)\eta dxdt\stackrel{n\to\infty}{\longrightarrow} 0.
\]
Using instead $d_n-d-M_n$, we show the reverse inequality, and we deduce
\[
\int_B \p(\nabla d)\eta dxdt\le\int_B (z\cdot\nabla d)\eta \, dxdt
\]
which concludes the proof.
\end{proof}

\section{Existence 
by approximation}\label{sec:genmob}

\subsection{A useful estimate: comparison with different forcing terms}\label{sec:forcing}

We prove in this section and the following a series of comparison
results, which will then be combined together to deduce a global
comparison result for flows with possibly different mobilities.
In this section, we shall assume that the surface
tensions $\p,\,\psi$ are smooth and elliptic, so that we can work in the classical viscosity setting.
In the limit, our main estimate will also hold for crystalline
flows in the sense of Definition~\ref{Defsol}.

We start by recalling standard comparison results for flows with
constant velocities, however we pay a special attention to the
particular metrics in which our velocities are expressed.
We first consider the equation
\begin{equation}\label{eq:nocurv}
u_t = \psi(\nabla u)g(x,t).
\end{equation}

The following result is a slight variant of 
the well known result~\cite[Theorem~8.1]{BarlesBook}:
\begin{lemma}\label{lem:lipcontrol}
Consider $u^0:\R^N\to \R$,
bounded and $\Lambda$-Lipschitz continuous with respect to a norm  $\eta$,
smooth and elliptic, 
such that
\begin{equation}\label{eq:controletopsi}
\psi\le\beta\eto. 
\end{equation}
Assume $g$ is bounded, continuous and  $M$-Lipschitz in space in 
the norm $\eta$.
Let $u(x,t)$ be a viscosity solution of~\eqref{eq:nocurv} with
initial datum $u_0$.
Then for all $t\ge 0$, the function $u(\cdot,t)$
is $\Lambda e^{\beta Mt}$-Lipschitz continuous in the norm $\eta$.
\end{lemma}
\begin{proof}
We start by observing that by classical results the solution $u$ is uniformly continuous, see for instance \cite{GGIS91}.
The rest of the  proof is an adaptation of the argument in  \cite[proof of Theorem 8.1]{BarlesBook}.
Let $\delta>0$ be given, and let $C$ be a smooth function such
that 
\begin{equation}\label{cprimo}
C' -\beta M C\ge \beta M\delta>0,
\end{equation}
with $C(0)=\Lambda$.
Set
\[
\sigma:= \sup_{\begin{subarray}{c}x,y\in\R^N\\t\in [0,T] \end{subarray}}
 u(x,t)-u(y,t)-C(t)\eta(x-y).
\]
We claim that $\sigma=0$. Using this claim, 
we have that
\[
u(x,t)-u(y,t)\le (\Lambda e^{\beta Mt}+\delta (e^{\beta Mt}-1))\eta(x-y)
\]
for all $x,y,t\le T$, and sending $\delta\to 0$ we
conclude the proof of the lemma.

We are left to prove the claim that $\sigma=0$. Arguing by contradiction, assume that $\sigma>0$. 
Consider a  maximum  point $( \bar x, \bar y, \bar t,  \bar s)$ 
in $\R^{2N}\times [0,T]^2$ for the function 
\[\vp(x,y,s,t) = u(x,t)-u(y,s)-C(t)\eta(x-y)-\frac{|t-s|^2}{2a}
-b\frac{|x|^2+|y|^2}{2},
\]
where $a,b>0$ are small parameters (notice  that
$\vp(x,y,0,0)\le 0$). 
For $b$ small enough, then  $\vp(\bar x,\bar y,\bar t,\bar s)\ge\sigma/2>0$, and
then by standard arguments (using in particular 
that $|\bar x|,|\bar y|\le c/\sqrt{b}$, and that for fixed $b$, both $\bar t$ and $\bar s$ converge, up to a subsequence, 
to the same positive value as $a\to 0$, see
for instance~\cite[Lemma~5.2]{BarlesBook})
we may assume $0< \bar t,\bar s \le T$, so that:
\begin{align*}
& C'(\bar t)\eta(\bar x-\bar y) +\frac{\bar t-\bar s}{a}\le \psi(C(\bar t)\nabla \eta(\bar x-\bar y)+b\bar x)g(\bar x,\bar t)\,,\\
&  \frac{\bar t-\bar s}{a}\ge \psi(C(\bar t)\nabla \eta(\bar x-\bar y)-b\bar y)g(\bar y,\bar s)\,.\\
\end{align*}
Evaluating the difference and recalling \eqref{cprimo} we obtain:
\[
\beta M(C(\bar t)+\delta)\eta(\bar x-\bar y)\le
\psi(C(\bar t)\nabla\eta(\bar x-\bar y)+b\bar x)g(\bar x,\bar t)
-\psi(C(\bar t)\nabla \eta(\bar x-\bar y)-b\bar y)g(\bar y,\bar s).
\]
For fixed $b>0$, we can then let $a\to 0$  and denote by 
$\tilde t\in (0,T]$ subsequence) the  common limit of $\bar t$ and $\bar s$ as $a\to 0$, and by  $\tilde x$ and $\tilde y$ the limit of $\bar x$ and $\bar y$ respectively. Thus,
using~\eqref{eq:controletopsi}, we obtain
\begin{multline*}
M(C(\tilde t)+\delta)\eta(\tilde x-\tilde y)
\\
\le \frac{1}{\beta}\psi(C(\tilde t)\nabla\eta(\tilde x-\tilde y)+b\tilde x)g(\tilde x,\tilde t)
-\frac{1}{\beta}\psi(C(\tilde t)\nabla \eta(\tilde x-\tilde y)-b\tilde y)g(\tilde y,\tilde t)\\
\le 
\frac{1}{\beta} (\psi(C(\tilde t)\nabla\eta(\tilde x-\tilde y)+b\tilde x)
- \psi(C(\tilde t)\nabla \eta(\tilde x-\tilde y)-b\tilde y))g(\tilde y,\tilde t)
\\+ \eto(C(\tilde t)\nabla\eta(\tilde x-\tilde y)+bx) M\eta(\tilde x-\tilde y).
\end{multline*}
We deduce
\begin{multline*}
C(\tilde t)+\delta\le \eto(C(\tilde t)\nabla\eta(\tilde x-\tilde y)+b\tilde x)
\\+ \frac{\psi(C(\tilde t)\nabla\eta(\tilde x-\tilde y)+b\tilde x)
-\psi(C(\tilde t)\nabla \eta(\tilde x-\tilde y)-b\tilde y)}{\beta M\eta(\tilde x-\tilde y)}\|g\|_\infty,
\end{multline*}
and sending $b\to 0$ (and observing that $\eta(\tilde x-\tilde y)\not\to 0$  
as $\sigma>0$ and $u$ is uniformly continuous), we find that if $\hat t$ is a limit point 
of $\tilde t$, then $C(\hat t)+\delta\le C(\hat t)$,  which gives a contradiction. 
Hence one must have $\sigma=0$. 
\end{proof}

In the next lemma we show that  if $E^0\subset F^0$ are initial sets and $1-\chi_E$, $1-\chi_F$
are viscosity solutions of~\eqref{eq:nocurv}, starting  from
$1-\chi_{E^0}$ and $1-\chi_{F^0}$, respectively, then  $\dist^{\eta}(\partial E(t),\partial F(t))\ge 
\dist^{\eta}(\partial E^0,\partial F^0)e^{-\beta Mt}$.

A splitting strategy will then extend this result to the solutions of~\eqref{eq:viscoflow}.

\begin{lemma} \label{lem:compgvisc}
Let $\eta$ be a smooth and elliptic norm satisfying
\eqref{eq:controletopsi}.
Let $g_1$, $g_2$ be two admissible forcing terms satisfying assumptions H1), H2) of 
Section~\ref{stass}, and both $M$-Lipschitz in the $\eta$ norm.
Assume
\begin{equation}\label{gmg}
g_2-g_1\le  c<+\infty \textup{ in } \R^N\times[0,+\infty).
\end{equation}
Let $E^0\subset F^0$ be two closed sets with $\dist^{\eta}(E^0,\R^N\setminus F^0):=\Delta >0$. Assume that $1-\chi_{E(t)}$ is a viscosity supersolution of
$u_t = \psi(\nabla u)g_1(x,t)$ starting from $1-\chi_E^0$,
and $1-\chi_{F(t)}$ a subsolution of $v_t=\psi(\nabla v)g_2(x,t)$ starting
from $1-\chi_F^0$. Then at all time $t\ge 0$,
\begin{equation}\label{eq:stimg1g2}
\dist^{\eta}(E(t),\R^N\setminus F(t)) \ge
\Delta e^{-\beta M t} +c\frac{1-e^{- \beta Mt}}{M}.
\end{equation}
\end{lemma}

\begin{proof} With Lemma~\ref{lem:lipcontrol} at hand, this is a straight
application of standard comparison principles.
We consider first
$u_0(x):=-\Delta\vee(2\Delta\wedge d^\eta_E(x))$ and $v_0(x):=-2\Delta\vee(\Delta\wedge d^\eta_F(x))$, so that
$v_0+\Delta\le u_0$. These functions are both $1$-Lipschitz in the
norm $\eta$.
We then consider the viscosity solutions $u$ of 
$u_t = \psi(\nabla u)g_1(x,t)$ starting from $u_0$,
and $v$ of $v_t=\psi(\nabla v)g_2(x,t)$, starting from $v_0$.
By standard comparison results, $E(t)\subseteq \{u(t)\le 0\}$ and
$F(t) \supseteq \{v(t)\le 0\}$, for all $t\ge 0$.

Thanks to Lemma~\ref{lem:lipcontrol},
$u(\cdot,t), v(\cdot,t)$ are $e^{\beta Mt}$-Lipschitz.
Let now $w(\cdot,t)= v(\cdot,t)+\Delta - c(e^{\beta Mt}-1)/M$, then at $t=0$,
$w(\cdot,0)=v_0+\Delta\le u_0$. We show that $w$ is a subsolution
of $u_t=\psi(\nabla u)g_1(x,t)$, so that $w\le u$. Indeed, if
$\vp$ is a smooth test function and $(\bar x,\bar t)$ a point
of maximum of $w-\vp$, then it is a point of maximum of
$v- (\vp-\Delta + c\beta (e^{\beta Mt}-1)/M)$ so that, 
using \eqref{gmg} and the fact that $v$ is a subsolution, we get
\[
\partial_t\vp(\bar x,\bar t) + c\beta e^{\beta M\bar t}
\le \psi(\nabla\vp(\bar x,\bar t))g_2(\bar x,\bar t)
\le \psi(\nabla\vp(\bar x,\bar t))g_1(\bar x,\bar t)+c\psi(\nabla\vp(\bar x,\bar t)).
\]
Since $\bar x$ is a contact point of the smooth function
$\vp(\cdot,\bar t)$ and the $e^{\beta M\bar t}$-Lipschitz function
$w(\cdot,\bar t)$
(in the $\eta$ norm), then $\eto(\nabla\vp)\le e^{\beta M\bar t}$
at $(\bar x,\bar t)$.
By~\eqref{eq:controletopsi}, $c\psi(\nabla \vp(\bar x,\bar t))\le c \beta e^{\beta M \bar t}$, whence
\[
\partial_t\vp\le \psi(\nabla\vp)g_1 
\]
and this shows that $w$ is a subsolution of this equation, hence
that $w\le u$. Therefore, 
for all $x,t$, $v(x,t)\le u(x,t)-\Delta + c(e^{\beta Mt}-1)/M$.
Thus, for $t\ge 0$ and $x,y\in\R^N$, 
recalling that $v$ is $e^{\beta M\bar t}$-Lipschitz,
\[
v(y,t)\le u(x,t) -  e^{\beta M t}
\left(\Delta e^{-\beta M t} - c\frac{1-e^{-\beta M t}}{M} -\eta(x-y)\right).
\]

It follows that if $\dist^\eta(y,E(t))\le \Delta e^{-\beta M t}-c(1-e^{-\beta M t})/M$, then $v(y,t)\le 1$, and hence $y\in F(t)$, which shows the lemma.
\end{proof}

\subsection{Comparison for different mobilities}\label{sec:mobilities}

In this section we provide the crucial stability estimates with respect to varying mobilities,
not necessarily smooth and elliptic.

\subsubsection{A comparison result with a constant forcing term}
In this subsection we shall assume that $\p,\psi_1,\psi_2$ are smooth and elliptic, and
 that
 \begin{equation}\label{eq:psivicino}
(1-\delta) \psi_2(\xi) \le \psi_1(\xi) \le (1+\delta)\psi_2(\xi) \qquad \text{ for all } \xi\in \R^N,
 \end{equation}
 for some (small) $\delta>0$.
We first show the following: 
\begin{lemma}\label{lem:compmobivisc}
There exists a constant $c_0>0$ depending only on  
$N$ such that the following holds:
let $\Delta>0$, and
let $E$ be a superflow for the equation $V=-\psi_1(\nu)\kappa_\p$ and
$F$ be a subflow for the equation $V=-\psi_2(\nu)(\kappa_\p-c_0\delta/\Delta)$,
with  $\dist^{\po}(E(0),\R^N\setminus F(0))= \Delta$.
Then for all $t$ until extinction of $E$ or $F^c$, there holds $\dist^{\po}(E(t),\R^N\setminus F(t))\ge \Delta$.
\end{lemma}

\begin{proof}
We first assume that $\partial E(t),\partial F(t)$
are bounded for all $t$.

We shall use the fact that $u(x,t)=-\chi_{E}(x,t)$ is a viscosity supersolution of
\begin{equation}\label{eq:visco1}
\partial_t u = \psi_1(\nabla u)\Div \nabla \p(\nabla u),
\end{equation}
while $v(x,t)=-\chi_{F}(x,t)$ is a viscosity subsolution of (see Lemma~\ref{lem:visco})
\begin{equation}\label{eq:visco2}
\partial_t v = \psi_2(\nabla v)\left(\Div \nabla \p(\nabla v) - c_0 \frac{\delta}{\Delta}\right).
\end{equation}

A first remark is that since the equations are translationally invariant, then also
\[
u'(x,t) = \inf_{\po(z)\le \Delta/4} u(x+z,t)
\]
is a supersolution of~\eqref{eq:visco1}, and similarly,
\[
v'(x,t) = \sup_{\po(z)\le \Delta/4} v(x+z,t)
\]
is a subsolution of~\eqref{eq:visco2}.
Remark that $u'= - \chi_{E'}$ and $v'=-\chi_{F'}$, with the tubes $E'$,
$F'$ defined by
\[
E'(t) = E(t)+W^\p(0,\tfrac{\Delta}{4})\,, \qquad
\R^N\setminus F'(t) = (\R^N\setminus F(t)) + W^\p(0,\tfrac{\Delta}{4})
\]
until their respective extinction time. We denote $t^*$ the minimum
extinction time of these sets.
In particular, 
\[
\dist^{\po}\left(E'(0),\R^N\setminus F'(0)\right)= \frac{\Delta}{2}.
\]
Using \cite[Lemma~2.6]{CMNP17}, 
there is a time $t_0$ such that for $t\le t_0$,
\[
\dist^{\po}\left(E'(t),\R^N\setminus F'(t)\right)\ge \frac{\Delta}{4}.
\]

Let $\e>0$,
and consider a point $(\bar x,\bar t,\bar y,\bar s)$
(depending on $\e$) which reaches
\begin{equation}\label{eq:Me}
M_\e=\min_{\newatop{x,\, y \in \R^N}{0\le s,\,t <t_0}}
\frac{1}{\e}(1+u'(x,t)-v'(y,s)) + \frac{\po(x-y)}{2}^2 + \frac{(t-s)^2}{2\e}
+ \frac{\e}{t_0-t} + \frac{\e}{t_0-s}.
\end{equation}
Observe that for every fixed $x\in E'(0)$, $y\not\in F'(0)$ and $s=t=0$,
this quantity is less than 
\[
\frac{\po(x-y)}{2}^2 + 2\frac{\e}{t_0}
\]
and in particular, $M_\e\le \Delta^2/8+ 2\e/t_0$. If $\e$
is small enough, one must have $1+u'(\bar x,\bar t)-v'(\bar y,\bar s)=0$,
that is, $\bar x\in E'(\bar t)$ and $\bar y\not\in F'(\bar s)$,
hence
$\po(\bar x-\bar y)=\dist^{\po}(E'(\bar t),\R^N\setminus F'(\bar s))$.

If both $\bar t,\bar s>0$, then from~\cite[Thm.~3.2]{CIL}
(with $\e=1$, in their notation),
there exist $(N+1)\times (N+1)$ symmetric matrices
\begin{equation}\label{XY}
\tilde X = \begin{pmatrix} X & \zeta \\ \zeta^T & \zeta_0 \end{pmatrix},\qquad
\tilde Y = \begin{pmatrix} Y & \eta \\ \eta^T & \eta_0 \end{pmatrix}
\end{equation}
such that
\begin{equation}\label{eq:subsuperjet}
\begin{cases}
(\nabla\po(\bar y-\bar x),\frac{\bar s-\bar t}{\e} - \frac{\e}{(t_0-\bar t)^2},
\tilde X) \in \overline{J^{2,-}\frac{u'}\e}(\bar x,\bar t)\,,\\
(\nabla\po(\bar y-\bar x),\frac{\bar s-\bar t}{\e} + \frac{\e}{(t_0-\bar s)^2},\tilde Y) \in \overline{J^{2,+}\frac{v'}\e}(\bar y,\bar s)\,,
\end{cases} 
\end{equation}
and such that
\begin{equation}\label{eq:ishii}
-\left(
1+\|A\|\right) \Id
\le \begin{pmatrix}- \tilde X & 0 \\ 0 & \tilde Y \end{pmatrix}
\le A+
 A^2
\end{equation}
where
\[
A = \begin{pmatrix}
D^2\po(\bar x-\bar y) & 0 & -D^2\po(\bar x-\bar y)  & 0 \\
0 & \frac{1}{\e}-2\frac{\e}{(t_0-\bar t)^3} & 0 & -\frac{1}{\e} \\
-D^2\po(\bar x-\bar y) & 0 & D^2\po(\bar x-\bar y)  & 0 \\
0 & -\frac{1}{\e}& 0 & \frac{1}{\e} -2\frac{\e}{(t_0-\bar s)^3}
\end{pmatrix}.
\]
In particular, for all $\xi\in \R^N$, letting $\tilde \xi = (\xi,0,\xi,0)\in\R^{2N+2}$,
from \eqref{eq:ishii} and \eqref{XY} we get 
\[
-\xi^T X\xi + \xi^T Y\xi\le \tilde\xi^T A\tilde\xi + 
\tilde\xi^T A^2\tilde\xi = 0, 
\]
which gives the inequality
\begin{equation}\label{eq:XY}
X\ge Y.
\end{equation}

Recall that $u'/\e$ is a supersolution and $v'/\e$ is a subsolution. 
Thanks to
\eqref{eq:subsuperjet}, letting $p=\nabla\po(\bar y-\bar x)$
and $a=\frac{\bar s-\bar t}{\e} $, one has
\[
\begin{cases}
a - \frac{\e}{(t_0-\bar t)^2} \ge \psi_1(p)D^2\p(p):X\,,\\
a + \frac{\e}{(t_0-\bar s)^2} \le \psi_2(p)(D^2\p(p):Y - c_0\tfrac{\delta}{\Delta})\,,
\end{cases}
\]
yielding
\begin{equation}\label{eq:almostdone}
0<\frac{\e}{(t_0-\bar t)^2} +
\frac{\e}{(t_0-\bar s)^2}
\le 
\psi_2(p)(D^2\p(p):Y - c_0\tfrac{\delta}{\Delta}) - 
 \psi_1(p)D^2\p(p):X\,.
\end{equation}
Now, we observe that as $E'(\bar t)= E(\bar t)+W^\p(0,\Delta/4)$
and (necessarily) $\bar x\in \partial E'(\bar t)$, we find
that $(p,X)$ is also a subjet of $-\chi_{W^\p(x',\Delta/4)}$ for
some $x'\in E(\bar t)$ with $\po(\bar x-x')=\Delta/4$. In
particular, it follows that $D^2\p(p):X \le 4(N-1)/\Delta$.
In the same way, $D^2\p(p):Y \ge -4(N-1)/\Delta$ and using~\eqref{eq:XY},
we obtain
\begin{equation}\label{eq:XY2}
-4\frac{N-1}{\Delta} \le D^2\p(p):Y \le D^2\p(p):X \le 
4\frac{N-1}{\Delta}.
\end{equation}

Thanks to~\eqref{eq:psivicino} and~\eqref{eq:XY2},
\begin{multline*}
 -\psi_1(p)D^2\p(p):X\le  -\psi_2(p)D^2\p(p):X  + \delta\psi_2(p)|D^2\p(p):X|
\\
\le 
-\psi_2(p)D^2\p(p):X + 4(N-1)\frac{\delta}{\Delta}\psi_2(p),
\end{multline*}
So that~\eqref{eq:almostdone} and~\eqref{eq:XY} yield
\begin{multline*}
0<\psi_2(p)(D^2\p(p):Y - c_0\tfrac{\delta}{\Delta}) - 
 \psi_1(p)D^2\p(p):X
\\
=
\psi_2(p)(D^2\p(p):(Y-X) - c_0\tfrac{\delta}{\Delta}) + 
 (\psi_1(p) - \psi_2(p)) D^2\p(p):X
\\
\le \psi_2(p)\big(D^2\p(p):(Y-X) - (c_0-4(N-1))\tfrac{\delta}{\Delta}\big)
\le 0
\end{multline*}
as soon as $c_0\ge 4(N-1)$, yielding a contradiction.
\smallskip

We deduce that at least one of  $\bar t$ or $\bar s$
is zero; without loss of generality let us assume $\bar s=0$.
For any $t<t_0$, thanks to~\eqref{eq:Me} (choosing $s=t$),
if $\e$ is small enough one has
\[
 \frac{1}{2}\dist^{\po}(E'(t),\R^N\setminus F'(t))^2 
+ 2\frac{\e}{t_0-t}
\ge
\frac{1}{2}\dist^{\po}(E'(\bar t),\R^N\setminus F'(0))^2 
+ \frac{\bar t^2}{2\e}
+ \frac{\e}{t_0-\bar t} + \frac{\e}{t_0}
\]
from which we see, in particular, that $\bar t\to 0$ as $\e\to 0$.
Hence, in the limit $\e\to 0$, using also that $E$ is closed (see \cite[Remark~2.3]{CMNP17} for more details), we deduce
\begin{multline*}
 \frac{1}{2}\dist^{\po}(E'(t),\R^N\setminus F'(t))^2 \ge 
\liminf_{\bar t\to 0}\frac{1}{2}\dist^{\po}(E'(\bar t),\R^N\setminus F'(0))^2 
\\
\ge \frac{1}{2}\dist^{\po}(E'(0),\R^N\setminus F'(0))^2  = \frac{\Delta^2}{8}
\end{multline*}
which shows the thesis of the Lemma, until $t=t_0$ (thanks to
the continuity property~(b)). Starting again from $t_0$, we have
proven the Lemma for bounded sets (or sets with bounded boundary).
\smallskip

If $\partial E(0)$ or $\partial F(0)$ is unbounded, we proceed as
follows: we first consider, for $\e>0$,
the sets 
$$
E^\e_0:=E(0)+W^{\p}(0,\e), \qquad  F^\e_0=:       \R^N\setminus \big(    (\R^N\setminus F(0))+W^{\p}(0,\e) \big),
$$ 
which satisfy
 $\dist^{\po}(E^\e_0,\R^N\setminus F^\e_0)\ge \Delta-2\e$.

Then, for $R>0$, we consider the initial sets
$E^{\e,R}_0 =E^\e_0\cap B_R$ and 
and $F^{\e,R}_0=F^\e_0 \cap (B_R+W^{\p}(0,\Delta))$. The result holds for
the evolutions starting from these two sets, with the distance $\Delta-2\e$.
Hence in the limit $R\to \infty$,
it must hold for the (viscosity) evolutions starting from $E^\e_0$ and $F^\e_0$
(which are unique for almost all choice of $\e$).

By standard comparison results for discontinuous viscosity
solutions~\cite{Barles,BarlSouga98,BaSoSou},
it then follows that the superflow $E$ (which is also a viscosity superflow)
is contained in the evolution starting from $E^\e_0$, while
$F$ contains the evolution starting from $F^\e_0$ (the $\e$-regularization
has been introduced to avoid issues due to the possible non-uniqueness
of viscosity solutions).

We deduce that $\dist^{\po}(E(t),\R^N\setminus F(t))\ge \Delta-2\e$
for all $t$,
until extinction. Since this is true for any $\e>0$, the lemma is proven.
\end{proof}

\subsubsection{Comparison with a non-constant forcing term}



In this section we prove the  crucial stability estimate for motions corresponding to different but close mobilities. 
We start with the following:
\begin{proposition}\label{prop:stabsmooth}
Assume 
that $\p,\psi_1,\psi_2$
are smooth and elliptic, 
that  $\psi_1,\psi_2$  satisfy~\eqref{eq:psivicino}, and that
\begin{equation}\label{eq:compani0}
\psi_2(\xi) \le\beta\p(\xi) \qquad \text{ for all } \xi\in\R^N.
\end{equation}
 Let $E_0\subset F_0$ be a closed and an open set, respectively, such that
 $\dist^{\po}(E_0,\R^N\setminus F_0)=:\Delta>0$, and let  $E$, $F$ be  a closed and open ``tube''
in $\R^N\times [0,\infty)$, respectively, with
$E(0)=E_0$, $F(0)=F_0$, such  that $-\chi_E$ is a supersolution
of 
\begin{equation}\label{eq:evol1}
u_t = \psi_1(\nabla u)(\Div\nabla\p(\nabla u)+g)\,,
\end{equation}
and $-\chi_F$ is a subsolution of
\begin{equation}\label{eq:evol2}
u_t = \psi_2(\nabla u)(\Div\nabla\p(\nabla u)+g)\,.
\end{equation}
Then, there holds 
\begin{equation}\label{eq:viscousmobiestim}
\dist^{\po}(E(t),\R^N\setminus F(t))\ge 
\Delta e^{-\beta Mt } - \delta\frac{2c_0/\Delta + \|g\|_\infty}{M}(1-e^{-\beta Mt })
\end{equation}
as long as this quantity is larger than $\Delta/2$,
where $c_0$ is as in Lemma \ref{lem:compmobivisc} and $M$ is the Lipschitz constant of $g$ with respect to $\po$.
\end{proposition}

\begin{proof} 
In order to obtain the estimate, we combine the results
of Lemmas~\ref{lem:compmobivisc} and~\ref{lem:compgvisc}
(with $\eta=\po$),
together with a splitting result which follows from~\cite{BarlesSouganidis91}
(\textit{cf}~Example~1, see also~\cite{BarlesSplitting}).

As before, we may need to slightly perturb the initial sets,
considering rather $E^s_0 = E_0+W^{\p}(0,s)$ and
$F^s_0 = \R^N\setminus ( \R^N\setminus F_0 +W^{\p}(0,s))$, for a small
$s$ (which eventually will go to $0$).

Given $s>0$ small, we start with building,
for $\e>0$ given, the motions $u^\e(x,t)$, $v^\e(x,t)$ defined
as follows:
we let $u^\e(x,0)=-\chi_{E^s_0}$ and define recursively $u^\e$
for $j\ge 0$ as a viscosity solution of:
\[
\begin{cases}
\displaystyle
u^\e_t = 2\psi_1(\nabla u^\e)\Div \nabla \p(\nabla v^\e)& 2j\e< t\le 2j\e+\e,\\[2mm]
\displaystyle
u^\e_t = 2\psi_1(\nabla u^\e)\fint_{2j\e}^{2(j+1)\e} g(x,s)ds & 2j\e+\e < t \le 2(j+1)\e.
\end{cases}
\]
(In case of nonuniqueness, we select for instance
the smallest (super)solution,
corresponding to the largest set $E^\e(t)=\{u^\e=-1\}$.)
Similarly, we let $v^\e(x,0)=-\chi_{F^s_0}$ and let $v^\e(x,t)$
be the largest (sub)solution of:
\[
\begin{cases}
\displaystyle
v^\e_t = 2\psi_2(\nabla v^\e)\left(\Div \nabla \p(\nabla v^\e)
- 2c_0\tfrac{\delta}{\Delta}\right)
& 2j\e< t\le 2j\e+\e,\\[2mm]
\displaystyle
v^\e_t = 2\psi_2(\nabla v^\e)
\Big(\fint_{2j\e}^{2(j+1)\e} g(x,s)ds+2c_0\tfrac{\delta}{\Delta}\Big)
 & 2j\e+\e < t \le 2(j+1)\e,
\end{cases}
\]
where $c_0$ is as in Lemma~\ref{lem:compmobivisc}.
Thanks to~\cite{BarlesSouganidis91,BarlesSplitting}, as $\e\to 0$
each of these functions converges to the viscosity solution of
\eqref{eq:evol1} and~\eqref{eq:evol2}, respectively, 
starting from $-\chi_{E^\e_0}$ and $-\chi_{F^\e_0}$,
provided these solutions are uniquely defined, which is known
to be true for almost all $\e$ (in fact all but a countable set of values),
in which case it is also known that they are (opposite of) characteristic
functions.

We now show that we can estimate the distance between the corresponding
geometric evolutions, using  Lemmas~\ref{lem:compmobivisc}  and~\ref{lem:compgvisc}.

Let $\delta$ be as in \eqref{eq:psivicino}. 
A first observation is that for $j\ge 0$, if we consider on the
interval $[2j\e+\e,2(j+1)\e]$ the smallest solution 
$\tilde{u}^\e(x,t)$ of
\[
\tilde{u}^\e_t = 2\psi_2(\nabla \tilde{u}^\e)
\Big(\fint_{2j\e}^{2(j+1)\e} g(x,s)ds - \delta\|g\|_\infty\Big)\,,\quad
\tilde{u}^\e(\cdot,2j\e+\e)=u^\e(\cdot,2j\e+\e)
\]
then, since for any $p\in\R^N$, 
\[
\psi_1(p)\fint_{2j\e}^{2(j+1)\e} g(x,s)ds
\ge 
\psi_2(p)\fint_{2j\e}^{2(j+1)\e} g(x,s)ds - \delta\psi_2(p)\|g\|_\infty
\]
one has $\tilde{u}^\e(x,t)\le u^\e(x,t)$ for $2j\e+\e\le t\le 2(j+1)\e$,
and thus $E^\e(t)\subseteq \{\tilde{u}^\e(\cdot,t)=-1\}$.
Hence, Lemma~\ref{lem:compgvisc} yields that for $2j\e+\e\le t\le 2(j+1)\e$,
\begin{multline*}
 \dist^{\po}(E^\e(t),\R^N\setminus F^\e(t))
\ge 
\dist^{\po}( \{\tilde{u}^\e(\cdot,t)=-1\},\R^N\setminus F^\e(t)\})
\\\ge 
 \left(\dist^{\po}(E^\e(2j\e+\e),\R^N\setminus F^\e(2j\e+\e))
 - \frac{c}{M}\right)e^{-2\beta M (t-2j\e-\e)} + \frac{c}{M},
\end{multline*}
{ for $c=- \delta (2c_0/\Delta + \|g\|_\infty)$.} Note that here
we use the fact that the mobility $2\psi_2$ satisfies $2\psi_2\le
2\beta\p$, \textit{cf}~\eqref{eq:compani0}.

On the other hand, Lemma~\ref{lem:compmobivisc} yields that
for all $j\ge 0$
and $2j\e \le t\le 2j\e+\e$, then
\[
\dist^{\po}(E^\e(t),\R^N\setminus F^\e(t))
\ge 
\dist^{\po}(E^\e(2j\e),\R^N\setminus F^\e(2j\e))
\]
as long as $\dist^{\po}(E^\e(2j\e),\R^N\setminus F^\e(2j\e))\ge \Delta/2$.

In particular, denoting $d_j=\dist^{\po}(E^\e(2j\e),\R^N\setminus F^\e(2j\e))$,
one obtains by induction that
\[ d_{j+1} \ge (d_j - \tfrac{c}{M})e^{-2\beta M\e} + \tfrac{c}{M}
\ge (d_0 - \tfrac{c}{M})e^{-2\beta M (j+1)\e} + \tfrac{c}{M},
\]
as long as $d_j\ge \Delta/2$. In the limit, we find that,
letting $E^s(t) = \{u(\cdot, t)=-1\}$ and $F^s(t) = \{v(\cdot,t)=-1\}$
and recalling that $\dist^{\po}(E^s_0,\R^N\setminus F^s_0)\ge \Delta-2s$,
\begin{equation*}
\dist^{\po}(E^s(t),\R^N\setminus F^s(t))\ge 
(\Delta-2s) e^{-\beta Mt } - \delta\frac{2c_0/\Delta + \|g\|_\infty}{M}(1-e^{-\beta Mt })
\end{equation*}
as long as this quantity is larger than $\Delta/2$.

By comparison, it is clear that $E\subset E^s$ and $F^s\subset F$,
hence (letting eventually $s\to 0$), we deduce that \eqref{eq:viscousmobiestim} holds as long as 
the right-hand side is larger than $\Delta/2$.
\end{proof}

We are now ready to state and prove the main result of the section.
\begin{theorem}\label{T7}
Let $\psi_1,\psi_2$ and $\p$ satisfy~\eqref{eq:psivicino} and~\eqref{eq:compani0}. Assume also that   $\psi_1,\psi_2$ are $\p$-regular in the sense of Definition~\ref{def:phiregular}.
Let the forcing term  $g(x,t)$ be  continuous, bounded, and
spatially $M$-Lipschitz continuous with respect to the distance $\po$, and
denote by $E$  a superflow for $V=-\psi_1(\nu)(\kappa_\p+g)$
and by $F$ be a subflow for $V=-\psi_2(\nu)(\kappa_\p+g)$, both
in the sense of Definition~\ref{Defsol}. Finally, assume  that
$\dist^{\po}(E(0),\R^N\setminus F(0))\ge \Delta>0$. Then for all $t$,
\begin{equation}\label{eq:mobiestim}
\dist^{\po}(E(t),\R^N\setminus F(t))\ge 
\Delta e^{-\beta Mt } - \delta\frac{2c_0/\Delta + \|g\|_\infty}{M}(1-e^{- M \beta t })
\end{equation}
as long as this quantity is larger than $\Delta/2$.
\end{theorem}

\begin{proof}  Consider smooth, elliptic approximations
of $\psi_i$ ($i=1,2$), $\p$, denoted $\psi_i^n$, $\p^n$, such
that~\eqref{eq:psivicino}-\eqref{eq:compani0} hold also for $\psi_i^n$, $\p^n$
 (with  slightly larger constants $\delta$ and $\beta$ that, with a small abuse of notation, will not be relabeled)
and with $\psi_i^n-\e \p^n$ convex ($i=1,2$), that is,
$\psi_i^n$ are uniformly
$\p^n$-regular (see statement of \textit{cf}~Theorem~\ref{th:exist}).

Consider as before, for $s>0$ small, the initial sets $E_0^s:=E_0+W^{\p^n}(0,s)$
and $F_0^s:=\R^N\setminus [ (\R^N\setminus F_0) +W^{\p^n}(0,s)]$. As in Theorem~\ref{th:exist} we can  build    subflows $A_n^s$ and superflows $B_n^s$ for the evolution
$V=-\psi^n_1(\nu)(\kappa_{\p^n}+g)$ both starting from $E_0^s$, such that  $A_n^s\subset B_n^s$, and 
a subflow $A_n'^{s}$ and superflow $B_n'^s$ for the evolution
$V=-\psi^n_2(\nu)(\kappa_{\p^n}+g)$ both starting from $F_0^s$, such that  $A_n'^s\subset B_n'^s$.
Thanks to Lemma~\ref{lem:visco},
$-\chi_{B_n^s}$ is a viscosity supersolution and $-\chi_{A_n'^s}$ is a viscosity
subsolution, so that we can apply Proposition~\ref{prop:stabsmooth} and estimate their $(\p^n)^\circ$-distance according to~\eqref{eq:viscousmobiestim}.

Againg thanks to Theorem~\ref{th:exist},
 $\R^N\setminus A_n^s$  converges in the
Kuratowski sense as $n\to \infty$ to the complement of a subflow, which contains $E$
thanks to Theorem~\ref{thm:uniq}, and analogously 
 $B_n'^s$ converges to a superflow contained in $F$.
We deduce~\eqref{eq:mobiestim} letting $s\to 0$.
\end{proof}

\subsection{ Existence and uniqueness by approximation}\label{sec:ex}

We recall that  the existence theory  for level set flows (in the sense of Definition~\ref{deflevelset1}) that we have so far works only for $\phi$-regular mobilities. The goal of  this section is to extend the existence theory to general mobilities. To this aim, we consider the following notion of {\em solution via approximation}: 
 
\begin{definition}[Level set flows via approximation]\label{deflevelset2}
 Let $\psi$, $g$, and $u^0$ be a  mobility,  an admissible forcing term, and a uniformly continuous function on $\R^N$, respectively. 
 
 We will say that a  continuous function $u^\psi:\R^N\times [0, +\infty)\to \R$ is a {\em solution via approximation} to the level set flow corresponding to \eqref{oee}, with initial datum $u^0$,  if $u^\psi(\cdot, 0)=u^0$ and if there exists  a sequence $\{\psi_n\}$ of $\phi$-regular mobilities  such that $\psi_n\to\psi$  and, denoting by $u^{\psi_n}$ the unique solution to \eqref{oee} (in the sense od Definition~\ref{deflevelset1}) with mobility $\psi_n$ and initial datum $u^0$, we have $u^{\psi_n}\to u^\psi$   locally uniformly in $\R^N\times[0,+\infty)$.
 \end{definition}

The next theorem is the main result of this section: it shows that for any mobility $\psi$, a solution-via-approximation $u^\psi$ in the sense of the previous definition always exists; such a solution  is also unique in that it is   independent of the choice of the approximating sequence of $\p$-regular mobilities $\{\psi_n\}$.
In particular, in the case of a $\p$-regular mobility, the notion of solution via approximation is consistent with that of Definition~\ref{deflevelset1}. 
\begin{theorem}\label{th:maingenmob}
Let $\psi$, $g$, and $u^0$ be as in Definition~\ref{deflevelset2}. Then, there exists a unique solution $u^\psi$ in the sense of Definition~\ref{deflevelset2} with initial datum $u^0$. 
\end{theorem}
\begin{proof}
We have to prove that for any sequence  $\{\psi_n\}$ of $\phi$-regular mobilities  such that $\psi_n\to\psi$, the corresponding solutions $u^{\psi_n}$  to \eqref{oee} with initial datum $u^0$ converge to some function $u$  locally uniformly  in $\R^N\times[0,+\infty)$.
We split the proof of the theorem into two steps.

{\it Step 1.} 
Let $\beta$ be as in \eqref{eq:compani0}. Let $T_0>0$ be defined by $e^{-2\beta M T_0} = \frac34$, where as usual $M$ is the spatial Lipschitz constant of the forcing term $g$ with respect to the 
distance induced by $\po$.
We claim that for every $\e>0$ there exists $\bar n \in \N$  such that 
\beq\label{step1eq}
\|u^{\psi_n}-u^{\psi_m}\|_{L^{\infty}(\R^N\times[0,T_0])}\leq \e \qquad \text{ for all } n,  m\ge \bar n.
\eeq
To this aim, 
we observe that since  $\psi_n\to \psi$, for $n$ large enough
\beq\label{psihat}
\psi_n(\xi) \leq 2 \beta \phi(\xi) \text{ for all } \xi \in \R^N, 
\eeq 
and  there exists $\delta_j\to 0$ such that
\beq\label{psihat2}
(1-\delta_j) \psi_n \leq \psi_m \leq (1+\delta_j) \psi_n \qquad \text{ for all } m, n \ge j.
\end{equation}
Set 
$E^{\psi_n}_\lambda(t):= \{u^{\psi_n}(\cdot, t) \leq \lambda\}$, $F^{\psi_n}_\lambda(t):= 
\{u^{\psi_n}(\cdot, t) < \lambda\}$ and recall that  $E^{\psi_n}_\lambda$ is a superflow, while $F^{\psi_n}_\lambda$ is a subflow in the sense of Definition  \ref{Defsol}.

Let $\omega$ be a modulus of continuity for  $u^0$ with respect to $\po$ and recall that for any $\lambda\in \R$
$$
\dist^{\po}(E^{\psi_m}_\lambda(0), \R^N\setminus F^{\psi_n}_{\lambda+\e}(0)) = 
\dist^{\po}(\{ u^0\le \lambda\}, \{ u^0\ge \lambda + \e\}) \geq \omega^{-1}(\e).
$$

By \eqref{psihat}, \eqref{psihat2} and Theorem \ref{T7}, for all $n,m\ge j$ we have
$$
\dist^{\po}(E_{\lambda}^{\psi_m}(t), \R^N\setminus F_{\lambda+\e}^{ \psi_n}(t))  \geq \omega^{-1}(\e) e^{-2\beta Mt} - \delta_j
\frac{ 2c_0/  \omega^{-1}(\e) + \|g\|_\infty}{M} (1- e^{-2\beta Mt}),
$$
as long as the right-hand side  is larger than $\omega^{-1}(\e)/2$, that is, for all $t\in [0,T_0]$, provided $j$ is large enough.
In particular, for $n, m$ large enough $E_{\lambda}^{\psi_m}(t) \subset  F_{\lambda+\e}^{ \psi_n}(t)$ for all $t\in [0,T_0]$ which yields 
$$
u^{\psi_n}(\cdot, t)\leq u^{\psi_m}(\cdot, t)+\e \qquad\text{for all  $t\in [0, T_0]$.}
$$ 
By switching the role of $n$ and $m$ we  deduce  \eqref{step1eq}.

\noindent {\it Step 2.}  
First arguing as in the proof of Theorem \ref{th:exist} and using \eqref{psihat} we see that $\omega(e^{2\beta Mt} \cdot)$ is a spatial modulus of continuity for $u^{\psi_n}(\cdot,t)$ for all $n$.
Observe that from \eqref{step1eq} it follows that for $n,m$ large enough we have 
\[
E^{\psi_m}_\lambda(T_0)\subseteq E^{\psi_n}_{\lambda +\e}(T_0),
\]
which in turn implies
\begin{multline*}
\dist^{\po}(E^{\psi_m}_\lambda(T_0), \R^N\setminus F^{\psi_n}_{\lambda+2\e}(T_0)) \ge 
\dist^{\po}(E^{\psi_n}_{\lambda + \e}(T_0), \R^N\setminus F^{\psi_n}_{\lambda+2\e}(T_0))
\\ \geq \omega^{-1}(e^{2\beta MT_0} \e).
\end{multline*}
We can now argue as in the Step 1 to conclude that, for $n, \, m$ large enough 
\[
\|u^{\psi_n}-u^{\psi_m}\|_{L^{\infty}(\R^N\times[T_0,2T_0])}\leq 2\e.
\]
Therefore, by an easy iteration argument we can shown that, for every given $T>0$, the sequence $\{u^{\psi_n}\}$ is a Cauchy sequence in $L^\infty(\R^N\times [0,T])$. 
This concludes the proof of the theorem.
\end{proof}

We conlcude by recalling the following remarks, referring to \cite{CMNP17} for the details.
\begin{remark}[Stability]\label{remstab}
As a byproduct of the previous theorem, and a standard diagonal argument, we have the following stability property for solutions to \eqref{oee}: let  $\{\psi_n\}_{n\in\N}$ be a sequence of mobilities 
and $\phi_n$ a sequence of anisotropies such that
$\psi_n\to  \psi$ and $\phi_n\to  \phi$ as $n\to +\infty$, 
then $u^{\psi_n}$ converge to $u^{\psi}$ locally uniformly  in $\R^N\times[0,+\infty)$ as $h\to 0$
(where $u^{\psi_n}$ is the solution to \eqref{oee} with $\psi$ replaced by $\psi_n$ and 
$\phi$ replaced by $\phi_n$).
\end{remark}

\begin{remark}[Comparison with the Giga-Pozar solution]\label{rm:gigapozar}
When $\phi$ is purely crystalline and $g\equiv c$ for some $c\in \R$ the unique level set solution in the sense of Definition~\ref{deflevelset2} coincides with the viscosity solution constructed in \cite{GigaPozar, GigaPozar17}.  
\end{remark}

We also recall that when $g$ is constant, \eqref{oee} admits a phase-field approximation by means of anisotropic Allen-Cahn equation, see \cite[Remark 6.2]{CMNP17} for the details.
\smallskip

In the next theorem we recall the main properties of the level set solutions introduced in Definition~\ref{deflevelset2}.  In the statement of the theorem,  we will say that a uniformly continuous initial function $u^0$ is {\em well-prepared} at $\lambda\in\R$ if the following two conditions hold:
\begin{itemize}
\item[(a)] If $H\subset \R^N$ is a closed set such that $\dist(H, \{u_0\geq \lambda\})>0$, then there exists $\lambda'<\lambda$ such that $H\subseteq
\{u_0< \lambda'\}$;
\item[(b)] If $A\subset \R^N$ is an open  set such that  $\dist(\{u_0\leq \lambda\}, \R^N\setminus A)>0$,  then there exists $\lambda'>\lambda$ such that 
$\{u_0\leq \lambda'\}\subset A$.
\end{itemize}
Here dist$(\cdot,\cdot)$ denotes the distance function with respect to a given norm. Clearly, the properties stated in (a) and (b) above do not depend on the choice of such a norm.

\begin{remark}
Note that the above assumption of well-preparedness is automatically satisfied if the set $\{u_0\leq \lambda\}$ is bounded.
\end{remark}
\begin{theorem}[Properties of the level set flow]\label{th:proplevelset}
Let  $u^\psi$ be  a solution in the sense  Definition~\ref{deflevelset2}, with initial datum $u^0$. The following properties hold true:

{\rm (i) (Non-fattening of level sets)}  There exists a countable set $N\subset \R$ such that  for all  
$t\in [0,+\infty), \, \lambda \not \in N$
\beq\label{eq:nonfattening2}
\begin{array}{rcl}  
\{(x,t):\, u^\psi(x, t) < \lambda\} &=& \mathrm{Int\,}(\{(x,t)\,:\, u^\psi(x, t) \le \lambda\})\,,\vspace{5pt}\\ 
\overline {\{(x,t)\,:\, u^\psi(x, t) < \lambda\}} &=& \{(x,t)\,:\, u^\psi(x, t) \le \lambda\}.
\end{array}
\eeq

{\rm (ii) (Distributional formulation when $\psi$ is $\p$-regular)} If $\psi$ is $\phi$-regular, then $u^\psi$ coincides with the  distributional solution in the sense of Definition~\ref{deflevelset1}.

 {\rm (iii) (Comparison)} Assume that $u^0\leq v^0$ and denote the corresponding level set flows by $u^\psi$ and $v^\psi$, respectively. Then $u^\psi\leq v^\psi$.

 {\rm (iv) (Geometricity)} Let $f:\R\to \R$ be increasing and uniformly continuous. 
Then $u^\psi$ is a solution with initial datum $u^0$ if and only if $f\circ u^{\psi}$ is a solution with initial datum $f\circ u^0$.

 {\rm (v) (Independence of the initial level set function)} Assume that $u^0$ and $v^0$ are well-prepared at $\lambda$. If
$\{u^0<\lambda\}=\{v^0<\lambda\}$, then $\{u^\psi(\cdot, t)<\lambda\}=\{v^\psi(\cdot, t)<\lambda\}$ for all $t>0$. Analogously, if 
$\{u^0\leq \lambda\}=\{v^0\leq \lambda\}$, then  $\{u^\psi(\cdot, t)\leq \lambda\}=\{v^\psi(\cdot, t)\leq \lambda\}$ for all $t>0$.
\end{theorem}
For the proof  we refer to \cite[Theorem 5.9]{CMNP17}.
\smallskip

We conclude with a remark about conditions that prevent the occurrence of fattening.

\begin{remark}[Star-shaped sets, convex sets and graphs] It is well-known  \cite[Sec.~9]{Soner93} that for the motion without forcing,
strictly star-shaped sets do not develop fattening so that, in particular, their evolution is unique. 
The proof of this fact, given for instance
in~\cite{Soner93} for the mean curvature flow, works also for solutions in the sense of Definition~\ref{Defsol} when the mobility  $\psi$ is $\p$-regular, and in turn, by approximation, also for the {\em generalized motion} associated to level set solutions in the sense of Definition~\ref{deflevelset2}, when $\psi$ is general. Uniqueness also holds for motions with a time-dependent forcing
$g(t)$ \cite[Theorem~5]{BeCaChNo-volpres} as long as the set remains strictly
star-shaped.
This remark obviously applies to initial convex sets, which, in
addition, remain convex for all times,
as was shown in~\cite{BelCaChaNo,CaCha,BeCaChNo-volpres} with a spatially constant forcing term.\footnote{Convexity is preserved also with a spatially convex forcing term but uniqueness is not known in this case.}
The case of unbounded initial
convex sets was not considered in these references but can be easily addressed by approximation (and uniqueness still holds with the same proof).

In the same way, if the initial set $E_0=\{x_N\le v^0(x_1,\dots,x_{N-1})\}$ is
 the subgraph of a uniformly continuous functions $v^0$, and the forcing term
does not depend on $x_N$, then
one can show that fattening does not develop and $E(t)$ is still the subgraph of a uniformly continuous function for all $t>0$, as in the classical case~\cite{EckerHuisken,EvansSpruckIII} (see also~\cite{GigaGiga98} for the 2D crystalline case).
\end{remark}




\end{document}